\documentclass{article}
\usepackage{amsthm,amssymb}
\begin{document}

\def\ZZ{{\mathbb Z}}
\def\RR{{\mathbb R}}
\def\CC{{\mathbb C}}
\def\EE{{\mathbb E}}
\def\PP{{\mathbb P}}
\def\Var{{\rm Var\ }}
\def\ceil#1{\lceil#1\rceil}
\def\floor#1{\lfloor#1\rfloor}
\newtheorem{theorem}{Theorem}
\newtheorem{corollary}[theorem]{Corollary}
\newtheorem{proposition}[theorem]{Proposition}
\newtheorem{lemma}[theorem]{Lemma}

\title{The pebbling threshold spectrum and paths}
\author{David Moews\\
\small Center for Communications Research\\[-0.8ex] 
\small 4320 Westerra Court\\[-0.8ex] 
\small San Diego, CA 92121\\[-0.8ex]
\small USA\\[-0.8ex]
\small\tt dmoews@ccrwest.org}

\maketitle

{\bf Abstract.}  Given a distribution of
pebbles on the vertices of a graph, say that we can {\em pebble}
a vertex if a pebble is left on it after some sequence of moves, each of which
takes two pebbles from some vertex and places one on an adjacent vertex.
A distribution is {\em solvable} if all vertices are pebblable;
the {\em pebbling threshold} of a sequence of graphs
is, roughly speaking, the total number
of pebbles for which random distributions with that number of
pebbles on a graph in the sequence
change from being almost never 
solvable to being almost always solvable.  We show that any sequence of connected
graphs with strictly increasing orders always has some pebbling threshold which is 
$\Omega(\sqrt{n})$ and $O(2^{\sqrt{2 \log_2 n}} n/\sqrt{\log_2 n})$, 
and that it is possible to construct such a sequence of connected graphs
which has any desired pebbling threshold between these bounds.
(Here, $n$ is the order of a graph in the sequence.)
It follows that the sequence of paths,
which, improving earlier estimates, we show has pebbling threshold 
$\Theta(2^{\sqrt{\log_2 n}} n/\sqrt{\log_2 n})$,
does not have the greatest possible pebbling threshold.

\section*{Introduction}

In the mathematical game of {\em pebbling}, one starts with a {\em distribution}
on a graph assigning a nonnegative integral number of pebbles to each
vertex of the graph.  A {\em pebbling move} consists of taking two pebbles
away from a vertex with at least two pebbles and adding one pebble to any 
adjacent vertex.  A vertex is {\em pebblable} for a given distribution
if there is some sequence of pebbling moves starting at the distribution
and finishing with at least one pebble on
that vertex, and a distribution is {\em solvable} if each vertex is
pebblable for that distribution.  In \cite{czy2002}, Czygrinow et al.~introduce 
the {\em pebbling threshold} for a sequence of graphs, which, 
roughly speaking, is the number of pebbles at which a random distribution
with that number of pebbles on a graph in the sequence
changes from being almost always unsolvable
to being almost always solvable.  

In \cite[RP15]{hurl2014}, Hurlbert asks
for the pebbling threshold of the sequence of paths.
In this paper, we 
determine that it is $\Theta(2^{\sqrt{\log_2 n}} n/\sqrt{\log_2 n})$, 
where $n$ is the number of vertices of a path in the sequence.
This makes more precise the estimates of \cite{bek2003}, \cite{czy2002}, \cite{czy2008},
and \cite{wier2004}.  We also prove some subsidiary results that may be of interest.
To do this, we first (\S 1) define uniform and
geometric probability distributions over multisets and the corresponding
thresholds of sequences of families of multisets.  
In \S 2, we improve some estimates used in \cite{bek2003},
and in \S 3, we use this to relate the
uniform and geometric thresholds.  In \S 4, we begin to compute the
pebbling threshold of the sequence of paths, relating it to a certain hypoexponential
distribution.  In \S 5, we estimate asymptotically some probabilities of this
distribution and finally complete the computation of the pebbling threshold
of the sequence of paths in \S 6.

The pebbling threshold of the sequence of paths is not the largest possible pebbling threshold.
The reason is that most vertices in the path can be moved onto from 
both directions; the ends are harder to reach since they can only be reached from one direction, 
but there are only two ends.  A graph which contains a bouquet of paths joined at a point will then 
be harder to pebble since it has more ends (for an appropriate choice of path lengths and number of paths.)  
In \S 7 and \S 8, we analyze a construction of this type; in \S 9, we 
show that any sequence of connected graphs with strictly increasing orders
has some pebbling threshold which is
$\Omega(\sqrt{n})$ and $O(2^{\sqrt{2 \log_2 n}} n/\sqrt{\log_2 n})$,
and also conversely show that any positive function which is
$\Omega(\sqrt{n})$ and $O(2^{\sqrt{2 \log_2 n}} n/\sqrt{\log_2 n})$
is the pebbling threshold of a sequence of connected graphs with
orders 1, 2, 3, {\dots}.
Here, $n$ is the order of a graph in the sequence.
This resolves the problem \cite[RP17]{hurl2014}.

\section{Definitions and notation}

We use 
$\ZZ$, 
$\ZZ_{>0}$, 
$\omega$,
$\RR$,
$\RR_{\ge 0}$, 
$\RR_{>0}$, 
$\RR_{<0}$, 
$\CC$,
$\PP$, $\EE$, and $\Var$ to denote 
the integers, 
the positive integers, 
the nonnegative integers, 
the reals,
the nonnegative reals,
the positive reals,
the negative reals,
the complex numbers,
probability, expectation, and variance.  $\iota$ is the imaginary unit.
For $x\in\RR$,
$\floor{x}$ will be the largest integer
no larger than $x$, $\ceil{x}$ the smallest integer no smaller than $x$,
and $\{x\}$ the fractional part of $x$, $\{x\}:=x-\floor{x}$;
for $x\in\RR_{>0}$, $\log x$ will be the natural logarithm of $x$,
and $\log_2 x$ will be the logarithm of $x$ to the base 2, $\log_2 x:=
\log x/(\log 2)$.
The cardinality of a set $S$ is written $\#S$.
For nonnegative integers $k\le n$, ${n\choose k}$ denotes the binomial coefficient $n!/(k!\,(n-k)!)$.
$A^B$ will be the set of functions from $B$ 
to $A$.  If $f$, $g\in A^B$ and $A$ is ordered, we define
$f\le g$ iff $f(x)\le g(x)$ for all $x\in B$; similarly, if $A$ has an addition
operation, we define $f+g\in A^B$ by $(f+g)(x)=f(x)+g(x)$ for all $x\in B$. 
We call the elements of $\omega^B$ {\em multisets}
and write 0 for the empty multiset,
i.e., the element of $\omega^B$ whose value is always 0.

For any $b\in B$, we take $e_b\in \omega^B$
to have $e_b(c)=1$ if $b=c$, $e_b(c)=0$ if $b\ne c$.  For $S$ a subset of some 
$\omega^B$, we let $\partial S$ be 
$\{f\in \omega^B\mid f+e_b\in S \hbox{\ for some\ } b\in B\}$,
and if $B$ is finite and $T\in \omega$, 
we take $[S]_T$ to be $\{f\in S\mid \sum_{x\in B} f(x)=T\}$.

The geometric distribution on $\omega$ with parameter $0<p\le 1$ is the
probability measure $\chi$ with $\chi(\{n\})=p (1-p)^n$, where we take
$0^0=1$.  For $B$ finite and nonempty, if $T\in\omega$, we let $\mu_T$ be the probability measure on
$\omega^B$ that is uniform on $[\omega^B]_T$ and zero elsewhere, and if
$T\in\RR_{\ge 0}$, we let $\nu_T$ be
the probability measure on $\omega^B$ which is the product of $\# B$ copies of the
geometric distribution with parameter $(1+(T/\# B))^{-1}$.

If $(M_i)_{i\in \omega}$ is a sequence 
such that $\emptyset\ne M_i\subseteq \omega^{B_i}$ for each $i$, where each 
$B_i$ is a nonempty finite set, and each $M_i$ is an {\em upper set} ($x\in M_i$ and $x\le y$ 
implies that $y\in M_i$), then we define the {\em uniform threshold} of $(M_i)_{i\in\omega}$ to 
be the sequence $(T_i)_{i\in\omega}$, where each $T_i\in\omega$ is minimal such that
$\mu_{T_i}(M_i)\ge \frac{1}{2}$.  If in addition $0\notin M_i$ for each $i$,
we define the {\em geometric threshold} of $(M_i)_{i\in\omega}$ to be
the sequence $(T_i)_{i\in\omega}$, where each $T_i\in\RR_{>0}$ is the unique
$T_i$ satisfying $\nu_{T_i}(M_i)=\frac{1}{2}$.  (We will see in $\S 2$ and $\S 3$ that these
definitions are sensible.)

All graphs considered in this paper will be undirected and finite; also, we will
take the graph with no vertices to be disconnected.
For a given graph, $d(x,y)$ will mean the distance between vertices $x$ and $y$ in the graph, and,
in a given distribution on the graph, ${\cal Z}(x)$ will mean the number of pebbles on vertex $x$.

\section{Improved thresholds for multisets}

The main result of this section is the following, which can be used to improve estimates
like those used in the proof of \cite[Theorem 1.5]{bek2003}.

\begin{theorem} \label{thm1}
If $B$ is nonempty and finite, $T\in\omega$, $x\in \RR_{\ge 0}$, 
$S\subseteq \omega^B$, and $\mu_{T+1}(S)\ge x/(T+1+x)$, then $\mu_T(\partial S)\ge x/(T+x)$.
(Here we take $0/0=0$ in the case where $T=x=0$.)
\end{theorem}

\begin{lemma}
\label{l1}
Given $x\in\RR_{\ge 0}$, $r\in\omega$ and positive integers $t$, $n$ and $d_0\ge d_1\ge \cdots \ge d_{r-1}$
with $t\ge r$ and (if $d_0$ exists) $n>d_0$, 
let
$$
p:=\left.\sum_{0\le i<r} {t-i-1+d_i \choose t-i}\right/{t-1+n\choose t},
$$
$$
q:=\left.\sum_{0\le i<r} {t-i-2+d_i \choose t-i-1}\right/{t-2+n\choose t-1}.
$$
Then $0\le p<1$ and, if $p\ge x/(t+x)$, also $q\ge x/(t-1+x)$ (where we take $0/0=0$ in the
case $t=1$ and $x=0$.)
\end{lemma}
\begin{proof}
We first prove that $0\le p<1$.  This is clear if $r=0$; otherwise
\begin{eqnarray*}
\sum_{0\le i<r} {t-i-1+d_i\choose t-i}
&<& \sum_{0\le j\le t} {j+d_0-1\choose j}\\
&=& {t+d_0\choose t}\\
&\le& {t-1+n\choose t},
\end{eqnarray*}
so $p<1$.

We can give an equivalent condition for $p\ge x/(t+x)$ implying that $q\ge x/(t-1+x)$
by observing that it will do to prove this for the maximal $x$ for which $p\ge x/(t+x)$,
which is $x:=pt/(1-p)$.  In this case, $q\ge x/(t-1+x)$ reduces to $(t-1+p)q\ge pt$.

We now induce on $t$ to prove this.  If $t=1$, we have two cases.  If $r=0$, we must have $p=0$,
making the result trivial; if $r=1$, $q=1$, making the result again trivial.
Otherwise, let $t>1$.  If $r=0$, we again have $p=0$, making the result trivial.  
If $r>0$, set
\begin{eqnarray*}
\alpha&:=&\sum_{1\le i<r}  {t-i-1+d_i \choose t-i},\\
\beta&:=&\sum_{1\le i<r} {t-i-2+d_i \choose t-i-1},\\
p'&:=&\alpha\left/{t-1+d_0\choose t-1}\right., \qquad
q':=\beta\left/{t-2+d_0\choose t-2}\right..
\end{eqnarray*}
By the induction hypothesis, we can assume that
$(t-2+p')q'\ge p'(t-1)$; since $0\le p'<1$, 
this implies that $p'\le q'$.
We need to show that $(t-1+p)q\ge pt$, which, after clearing denominators, is equivalent to
\begin{eqnarray}
&\ &\left((t-1){t-1+n\choose t} + p' {t-1+d_0\choose t-1} + {t-1+d_0\choose t}\right)\cdot
\nonumber\\
&\ &\left(q'{t-2+d_0\choose t-2}+{t-2+d_0\choose t-1}\right)\nonumber\\
\label{eqa}
&\ge&
{t-2+n\choose t-1}\left(p' {t-1+d_0\choose t-1} + {t-1+d_0\choose t}\right)t.
\end{eqnarray}
Taking a forward first difference of (\ref{eqa}) with respect to $n$ gives
\begin{eqnarray*}
&\ &(t-1){t-1+n\choose t-1}\left(q'{t-2+d_0\choose t-2}+{t-2+d_0\choose t-1}\right)\\
&\ge&
{t-2+n\choose t-2}\left(p' {t-1+d_0\choose t-1} + {t-1+d_0\choose t}\right)t,
\end{eqnarray*}
which, using ${t-1+n\choose t-1}=\frac{t+n-1}{t-1}{t-2+n\choose t-2}$
and removing the common factor ${t-2+n\choose t-2}$,
can be rewritten as
$$
(t+n-1)\left(q'{t-2+d_0\choose t-2}+{t-2+d_0\choose t-1}\right)\ge
\left(p' {t-1+d_0\choose t-1} + {t-1+d_0\choose t}\right)t.
$$
Since we assume $n>d_0$, it's enough to prove this when $n=d_0+1$.
Using ${t-1+d_0\choose t}=\frac{t-1+d_0}{t}{t-2+d_0\choose t-1}$, this simplifies to
$$
q' (t+d_0) {t-2+d_0\choose t-2} + {t-2+d_0\choose t-1}\ge
p' {t-1+d_0\choose t-1}t,
$$
and since $p'\le q'$ and $p'\le 1$, it's enough to show that
$$
(t+d_0) {t-2+d_0\choose t-2} + {t-2+d_0\choose t-1}
={t-1+d_0\choose t-1}t,
$$
which is easy as both sides are multiples of ${t-2+d_0\choose t-2}$ by rational
functions of $t$ and $d_0$.

It remains to prove (\ref{eqa}) when $n=d_0+1$.
In this case, looking at a portion of the left-hand side of (\ref{eqa}),
\begin{eqnarray*}
&\ &q'{t-2+d_0\choose t-2}\left((t-1){t+d_0\choose t}+p'{t-1+d_0\choose t-1}\right)\\
&=& q'{t-2+d_0\choose t-2}\left((t-2+p'){t-1+d_0\choose t-1} + {t+d_0\choose t}+\right.\\
&\ & \ \ \ \left.(t-2)\left({t+d_0\choose t}-{t-1+d_0\choose t-1}\right)\right)\\
&\ge& p'(t-1) {t-2+d_0\choose t-2} {t-1+d_0\choose t-1} +\\
&\ & \ \ \ q'{t-2+d_0\choose t-2}\left( {t+d_0\choose t}+(t-2) {t-1+d_0\choose t}\right).
\end{eqnarray*}
It's therefore enough to prove (\ref{eqa}) with the left-hand side replaced by
\begin{eqnarray*}
&\ & p'(t-1) {t-2+d_0\choose t-2} {t-1+d_0\choose t-1} +\\
&\ & \ \ \ q'{t-2+d_0\choose t-2}\left( {t+d_0\choose t}+(t-2) {t-1+d_0\choose t}\right)+\\
&\ & q'{t-2+d_0\choose t-2} {t-1+d_0\choose t} +\\
&\ & \left((t-1){t+d_0\choose t} + p' {t-1+d_0\choose t-1} + {t-1+d_0\choose t}\right)
{t-2+d_0\choose t-1}.
\end{eqnarray*}
Using $p'\le q'$ and separating terms which involve and do not involve $p'$, it will do
to show that
\begin{eqnarray*}
&\ & (t-1) {t-2+d_0\choose t-2} {t-1+d_0\choose t-1} +\\
&\ & \ \ \ {t-2+d_0\choose t-2}\left( {t+d_0\choose t}+(t-2) {t-1+d_0\choose t}\right)+
\\
&\ & {t-2+d_0\choose t-2} {t-1+d_0\choose t} + {t-1+d_0\choose t-1} {t-2+d_0\choose t-1}
\\ 
&=& {t-1+d_0\choose t-1}^2 t
\end{eqnarray*}
and
\begin{eqnarray*}
 \left((t-1){t+d_0\choose t} + {t-1+d_0\choose t}\right) {t-2+d_0\choose t-1}
= {t-1+d_0\choose t-1}{t-1+d_0\choose t}t.
\end{eqnarray*}
This can be done by expressing both sides of both of these equations as rational multiples
of ${t-2+d_0\choose t-2}{t-1+d_0\choose t-1}$.
\end{proof}

We now prove Theorem \ref{thm1}.  
\begin{proof} 
Set $n:=\#B$, and intersect $S$ with $[\omega^B]_{T+1}$ 
if necessary so we can 
assume that $S\subseteq [\omega^B]_{T+1}$.  If $S=[\omega^B]_{T+1}$, then $\partial S=[\omega^B]_T$,
so $\mu_T(\partial S)=1$ and the result is obvious; if $S$ is empty, then we must have $x=0$ so the result is again
obvious.  Otherwise, set $t:=T+1$.  Since we have $\emptyset\subsetneq S\subsetneq[\omega^B]_t$, we must have
$0 < \#S < \# [\omega^B]_t={t+n-1\choose t}$.  By the theorem in 
\cite{clem1984},
there is then a representation
$$\#S = \sum_{0\le i<r} {t-i-1+d_i\choose t-i} $$
with $t\ge r>0$, $d_0\ge d_1\ge\cdots \ge d_{r-1}>0$, and
$$\#(\partial S)\ge \sum_{0\le i<r} {t-i-2+d_i\choose t-i-1}.$$
Since $\#S < {t+n-1\choose t}$ we must then have $n>d_0$, so we can apply Lemma \ref{l1}.
\end{proof}

\begin{proposition}  \label{thm2}
If $B$ is finite and nonempty and
$M\subseteq \omega^B$ is an upper set, $T\ge U\in\omega$, $x\in\RR_{\ge 0}$,
and $\mu_T(M)\le T/(T+x)$, then $\mu_U(M)\le U/(U+x)$.   (Here we take $0/0=1$ if $U=x=0$.)
\end{proposition}
\begin{proof}  If $x=0$, the result is obvious.
Otherwise, since $\mu_T(M)\le T/(T+x)$, $\mu_T(\omega^B\setminus M)\ge x/(T+x)$,
so by repeated application of Theorem \ref{thm1}, $\mu_U(\partial^{T-U}(\omega^B\setminus M))\ge x/(U+x)$.
However, if $v\in \partial^{T-U}(\omega^B\setminus M)$, we have $v\le w$ for some
$w\in \omega^B\setminus M$, so we cannot have $v\in M$ since then, as $M$ is an upper set, $w$ would
also be in $M$.  Therefore, $\partial^{T-U}(\omega^B\setminus M)$ is disjoint from $M$
so $\mu_U(M)\le \mu_U(\omega^B\setminus \partial^{T-U}(\omega^B\setminus M))\le 1-(x/(U+x)) = U/(U+x)$.
\end{proof}

\begin{proposition}
\label{thm3}
If $B$ is finite and nonempty and $M\subseteq \omega^B$ is an upper set, $T$ and $U$ are positive integers
with $T\ge U$, $x\in\RR_{\ge 0}$, and $\mu_U(M)\ge U/(U+x)$, then $\mu_T(M)\ge T/(T+x)$.
\end{proposition}
\begin{proof}
Replace $x$ by $x+\epsilon$, apply the contrapositive of Proposition \ref{thm2}, and let $\epsilon\to 0$
from above.
\end{proof}

\begin{proposition}
\label{thm4}
If $B$ is finite and nonempty, $M\subseteq \omega^B$ is an upper set, $U\in\omega$, and
$\mu_U(M)>0$, then $\mu_U(M)\le\mu_{U+1}(M)$ and $\lim_{i\to\infty} \mu_i(M)=1$.
Also, if $0<\mu_U(M)<1$, then $\mu_U(M)<\mu_{U+1}(M)$.
\end{proposition}
\begin{proof}
If $U=0$, then we must have $0\in M$ so $M=\omega^B$ and $\mu_i(M)=1$ for all $i$.  
Otherwise, we can apply Proposition \ref{thm3} with some value of $x$ to conclude that
$\lim_{i\to\infty} \mu_i(M)=1$.  Applying it with the minimum possible value of $x$
allows us to conclude that $\mu_U(M)\le \mu_{U+1}(M)$, or $\mu_U(M)<\mu_{U+1}(M)$ in the case where
$x>0$, i.e., $\mu_U(M)<1$.
\end{proof}

\begin{theorem}
\label{thm5}
If $B$ is finite and nonempty and $M\subseteq \omega^B$ is an upper set, then the sequence
$(\mu_0(M), \mu_1(M), \ldots)$ is either:
\begin{enumerate}
\item $(0,0,\ldots)$, if $M$ is empty.
\item $(0,0,\ldots,0,r_0,\ldots,r_{N-1},1,1,\ldots)$, for some $N\in\omega$ and $0<r_0<\cdots<r_{N-1}<1$.
\item $(0,0,\ldots,0,r_0,r_1,\ldots)$, for some strictly increasing sequence $(r_i)_{i\in\omega}$ of positive
real numbers with $\lim_{i\to\infty} r_i=1$.
\end{enumerate}
(The initial sequence of zeroes may be empty in cases 2 and 3.)
\end{theorem}
\begin{proof}
Apply Proposition \ref{thm4} repeatedly.
\end{proof}

Theorem \ref{thm5} shows that the definition of uniform threshold makes sense.

\section{Uniform and geometric thresholds}

In this section, we show that the definition of geometric threshold is sensible; also, 
if both uniform and geometric thresholds of a sequence of multiset families
are defined, the geometric threshold approaches infinity, and the number of elements in the base sets of the
multiset families approaches infinity, then the two thresholds have asymptotic ratio 1.

\begin{theorem}
\label{thm6}
If $B$ is finite and nonempty and $M\subseteq \omega^B$ is an upper set, then
the function $x\mapsto \nu_x(M)$ on $\RR_{\ge 0}$ is either:
\begin{enumerate}
\item Identically 0, if $M$ is empty.
\item Identically 1, if $M=\omega^B$.
\item Strictly increasing and continuous with $\nu_0(M)=0$ and $\lim_{x\to\infty} \nu_x(M)=1$, otherwise.
\end{enumerate}
\end{theorem}
\begin{proof}
If $M$ is empty or $\omega^B$, this is obvious.  Assume otherwise.  Since $M\ne\omega^B$, $0\notin M$
so $\nu_0(M)=0$ and, since $\nu_x$ always assigns positive probability to 0, $\nu_x(M)<1$ for all $x$.

If $(G_b)_{b\in B}$ is an i.i.d.~family of geometric random variables, then, conditioned on
$\sum_b G_b=T$, the function $b\mapsto G_b$ is uniform on $[\omega^B]_T$.  Therefore,
$\nu_x(M)=\EE(\mu_{N_x}(M))$, where the random variable $N_x$ is the sum of $\# B$ i.i.d.~geometric random 
variables with parameter $(1+(x/\#B))^{-1}$.
Given a geometric random variable $G_p$ with parameter $p$, we can realize $G_p$ as the smallest $i$
for which an i.i.d.~sequence of random variables $U_0$, $U_1$, \dots uniform on $[0,1]$ has $U_i<p$.
This lets us realize $G_p$ and $G_q$ ($p<q$) on the same probability space with
$G_p=G_q+\Xi$, where $\Xi$ is a nonnegative integral random variable which, conditioned on $G_q$,
always assigns positive probability to each nonnegative integer; in fact, $\PP(\Xi=0\mid G_q)$ is
always $p/q$.
Summing $\# B$ independent copies of this, we can realize $N_x$ and $N_y$ ($x<y$) on the same probability 
space with $N_y=N_x+\Xi'$, where $\Xi'$ is a nonnegative integral random variable which, conditioned
on $N_x$, always assigns positive probability to each nonnegative integer; also, $\PP(\Xi'=0\mid N_x)$
is always $((x+\#B)/(y+\#B))^{\# B}$.
But then
\begin{eqnarray*}
\nu_y(M)&=& \EE(\mu_{N_y}(M))\\
&=& \EE(\EE(\mu_{N_y}(M)\mid N_x))\\
&=& \EE(\EE(\mu_{N_x+\Xi'}(M)\mid N_x)).
\end{eqnarray*}
By Theorem \ref{thm5},
$\EE(\mu_{N_x+\Xi'}(M)\mid N_x)\ge \mu_{N_x}(M)$. Also, by Theorem \ref{thm5} and
the above property of $\Xi'$, 
$\EE(\mu_{N_x+\Xi'}(M)\mid N_x)>\mu_{N_x}(M)$ whenever $\mu_{N_x}(M)<1$.
Since $\nu_x(M)<1$ we must have $\PP(\mu_{N_x}(M)<1)>0$, so
$\EE(\EE(\mu_{N_x+\Xi'}(M)\mid N_x))>\EE(\mu_{N_x}(M))=\nu_x(M)$.  This proves that
$x\mapsto \nu_x(M)$ is strictly increasing.  To show that it is continuous, observe that
\begin{eqnarray*}
\EE(\mu_{N_x+\Xi'}(M)\mid N_x)&\le& \PP(\Xi'=0\mid N_x) \mu_{N_x}(M) + 1 - \PP(\Xi' = 0\mid N_x)\\
&\le& \mu_{N_x}(M) + 1-(\frac{x+\#B}{y+\#B})^{\# B},
\end{eqnarray*}
and, taking expectations,
$$
0<\nu_y(M)-\nu_x(M)\le 1-(\frac{x+\#B}{y+\#B})^{\# B},
$$
which implies that $x\mapsto \nu_x(M)$ is continuous.

Using Theorem \ref{thm5} again, to prove 
that $\lim_{x\to\infty} \nu_x(M)=1$, it will do to prove that $\lim_{x\to\infty} \PP(N_x\le j)=0$
for each fixed $j\in\omega$.  This is so because
\begin{eqnarray*}
\PP(N_x\le j) &\le& \PP(G_{(1+(x/\#B))^{-1}}\le j) \\
&\le& (j+1) (1+(x/\#B))^{-1}\\
&\to& 0 \hbox{\ \ as $x\to\infty$.}
\end{eqnarray*}
\end{proof}

Theorem \ref{thm6} shows that the definition of geometric threshold makes sense.

\begin{proposition} \label{thm7}
If $B$ is finite and nonempty, $M\subseteq \omega^B$ is an 
upper set, $T\ge U\in\omega$, and $\mu_U(M)\ge \frac12$, then $\mu_T(M)\ge T/(T + U)$, where here we take $0/0=1$ if $T=U=0$.
\end{proposition}
\begin{proof}  If $U=0$, $M$ must contain 0, so $M=\omega^B$ and the result is obvious.  Otherwise, set $x:=U$ and use Proposition \ref{thm3}.
\end{proof}

\begin{proposition} \label{thm8}
If $B$ is finite and nonempty, $\emptyset\ne M\subseteq \omega^B$ is an 
upper set, $U<T\in\omega$, and $T$ is minimal such that $\mu_T(M)\ge \frac12$, then $\mu_U(M)\le U/(T + U - 1)$, 
where here we take $0/0=0$ if $U=0$ and $T=1$.
\end{proposition}
\begin{proof}  Since $T\ne 0$, $0\notin M$, so $\mu_0(M)=0$; this proves the result if $U=0$.  Otherwise, set $x:=T-1$ and use Proposition \ref{thm2} with $T$ decreased by 1.
\end{proof}

\begin{proposition} \label{thm9}
If $B$ is finite and nonempty, $\emptyset\ne M\subseteq \omega^B$ is an upper set with
$0\notin M$, $T\in\omega$ is minimal such that $\mu_T(M)\ge \frac12$, $T'$ is the unique positive real number such that $\nu_{T'}(M)=\frac 12$, $S:=\sqrt{T'+(T'^2/\#B)}$, and $\theta$ is a real number with $\sqrt{2}<\theta< T'/S$, then
\begin{equation}
\label{9conc}
\ceil{T'-\theta S}(1-\frac{2}{\theta^2})\le T\le 
1 + \floor{T' + \theta S}(1 + \frac{2}{\theta^2-2}).
\end{equation}
\end{proposition}
\begin{proof}
If $G_p$ is a geometric random variable with parameter $p$, then $\EE G_p = p^{-1}-1$ and
$\Var G_p=p^{-2}-p^{-1}$.  Therefore, if $N_{T'}$ 
is the sum of $\# B$ i.i.d.~geometric random variables with parameter $(1+(T'/\#B))^{-1}$,
then $\EE N_{T'}=T'$ and $\Var N_{T'}=T'+(T'^2/\#B)=S^2$.  If $V:=\ceil{T'-\theta S}$ and
$W:=\floor{T'+\theta S}$, it follows from Chebyshev's inequality that $$\PP(V\le N_{T'}\le W)\ge
1-\frac{1}{\theta^2},$$ so since, by Theorem \ref{thm5}, $\mu_i(M)$ is nondecreasing with $i$, 
\begin{equation}\label{91}
\frac 12 = \nu_{T'}(M)=\EE(\mu_{N_{T'}}(M))\ge (1-\frac{1}{\theta^2}) \mu_V(M)
\end{equation}
and
\begin{equation}\label{92}
\frac 12 = \nu_{T'}(M)=\EE(\mu_{N_{T'}}(M))\le (1-\frac{1}{\theta^2}) \mu_W(M)+\frac{1}{\theta^2}.
\end{equation}
If $V\le T$, then the left-hand inequality of (\ref{9conc}) is satisfied.  
If $T<V$, then by Proposition \ref{thm7} and (\ref{91}),
$$
\frac{\frac 12}{1-\theta^{-2}}\ge \mu_V(M)\ge \frac{V}{T+V},
$$
which, after rearrangement, gives the left-hand inequality of (\ref{9conc}).
If $T\le W$, then the right-hand inequality of (\ref{9conc}) is satisfied.
If $W<T$, then by Proposition \ref{thm8} and (\ref{92}),
$$
\frac{\frac 12-\theta^{-2}}{1-\theta^{-2}}\le \mu_W(M)\le \frac{W}{T+W-1},
$$
which, after rearrangement, gives the right-hand inequality of (\ref{9conc}).
\end{proof}

\begin{theorem} \label{thm10}
If $(M_i)_{i\in \omega}$ is a sequence
such that $\emptyset\ne M_i\subseteq \omega^{B_i}$ for each $i$, each
$B_i$ is a nonempty finite set, each $M_i$ is an upper set, $0\notin M_i$
for all $i$, $(T_i)_{i\in\omega}$ and $(T'_i)_{i\in\omega}$ are the uniform
and geometric thresholds of 
$(M_i)_{i\in\omega}$, and $\# B_i$ and $T'_i$ both approach infinity
as $i\to\infty$, then $T_i/T'_i\to 1$ as $i\to\infty$.
\end{theorem}
\begin{proof}
For sufficiently large $i$, use 
Proposition \ref{thm9} on each $B:=B_i$, $M:=M_i$, $T:=T_i$, and $T':=T'_i$, setting $$\theta:=
(\frac{T'}{S})^{1/3}=(\frac{T'}{\sqrt{T'+(T'^2/\# B)}})^{1/3}$$
and observing that since $T'/S\to\infty$ as $i\to\infty$,
$\theta$ and $T'/(\theta S)$ both approach
$\infty$ as $i\to\infty$.
\end{proof}

\section{The threshold of the sequence of paths, I}

We now begin to compute the pebbling threshold of the sequence of $n$-paths, where the $n$-path
has $n$ vertices, 1, \dots, $n$, and edges between vertices $i$ and $i+1$ for all $i=1$,
\dots, $n-1$.  Because of Theorem \ref{thm10} and \cite[Theorem 1.3]{bek2003},
it will do to find the geometric threshold of the 
sequence of families of solvable distributions of the $n$-paths.  Therefore, fix some positive $n$,
and suppose that, for some parameter $0<p<1$, we have i.i.d.~geometric random variables
$(Z_i)_{i\in\ZZ_{>0}}$ with parameter $p$, and that, for $i=1$, \dots, $n$,
we place $Z_i$ pebbles on each vertex $i$.
If $r:=(2 \log n)/p$, then $\PP(Z_i\ge r)=\PP(Z_i\ge \ceil{r})=(1-p)^{\ceil{r}}\le e^{-pr}=n^{-2}$,
so with probability at least $1-n^{-1}$, $Z_i<r$ for all $i=1$, \dots, $n$.
Let $L$ be a positive integer such that $n\ge 2L+1$.  Now, for each $L+1\le i\le n-L$, $i$ will be 
unpebblable iff $Z_i=0$, $\sum_{1\le j<i} Z_j 2^{-(i-j)} < 1$, and $\sum_{i<j\le n} Z_j 2^{-(j-i)}<1$,
so, given that $Z_j<r$ for all $j=1$, \dots, $n$, it is sufficient for unpebblability that
$Z_i=0$, $\sum_{1\le k\le L} Z_{i-k} 2^{-k} < 1 - (r/2^L)$, and
$\sum_{1\le k\le L} Z_{i+k} 2^{-k} < 1 - (r/2^L)$.
If we pick $i=L+1$, $L+1+(2L+1)$, \dots, $L+1+(\floor{\frac{n}{2L+1}}-1)(2L+1)$, then,
since the $Z_i$'s are i.i.d., the probability that the distribution is unsolvable will be at least
$$
-\frac{1}{n} + 1 - (1 - pq^2)^{\floor{n/(2L+1)}},
{\rm \ \ where\ \ }
q:=\PP(\frac{Z_1}{2}+\cdots+\frac{Z_L}{2^L}<1 - \frac{r}{2^L}),
$$
so the probability that it is solvable will be no more than
$$
\frac{1}{n}+(1-pq^2)^{\floor{n/(2L+1)}}\le \frac{1}{n}+\exp -pq^2\floor{\frac{n}{2L+1}}.
$$
It is easy to see that $\sum_{i>0} Z_i/2^i$ converges a.s., and then 
$$
q\ge q':= \PP(\sum_{i>0} \frac{Z_i}{2^i}<1-\frac{r}{2^L}).
$$
If we let $(W_i)_{i\in\ZZ_{>0}}$ be an i.i.d.~family of standard exponential random variables,
with $\PP(W_i\ge x)=e^{-x}$ for each $i$, then we can realize $Z_i$ by letting each $Z_i$
be $\floor{W_i/\lambda}$, where $\lambda:=-\log(1-p)$.  Since $\sum_{i>0} W_i/2^i$ also
converges a.s.,~we have
\begin{eqnarray*}
q'&=&\PP(\sum_{i>0} \frac{\floor{W_i/\lambda}}{2^i}<1-\frac{r}{2^L})\\
&\ge& q'':=\PP(\sum_{i>0} \frac{W_i/\lambda}{2^i}<1-\frac{r}{2^L})\\
&\ &\ \ \ \ \,\, =\PP(Y_{\infty}<2\lambda(1-\frac{r}{2^L})),
\end{eqnarray*}
where we have set
$$
Y_{\infty}:=W_1 +\frac{W_2}{2}+\frac{W_3}{4}+\cdots=\sum_{i\ge 0} \frac{W_{i+1}}{2^i}.
$$
The probability that the distribution is solvable is then no more than
$$
\frac{1}{n}+\exp -pq''^2 \floor{\frac{n}{2L+1}}.
$$
To estimate this, we must estimate the probability that $Y_\infty$
is below a small threshold.

\section{The asymptotics of $Y_\infty$}

Let $n\in\ZZ_{>0}$, $Y_n:=W_1+\cdots+(W_n/2^{n-1})$, 
and, for $i=0$, \dots, $n-1$, let
$$
R_{i,n}(x):=\prod_{0\le j\le n-1, j\ne i} \frac{2^j-x}{2^j-2^i}
$$
be the degree $n-1$ polynomial which is 0 at $2^j$, $j=0$, \dots, $n-1$, $j\ne i$,
and 1 at $2^i$.
Then for all $x\in\RR_{\ge 0}$
\cite[\S I.13, ex.~12]{fel1971},
$$
\PP(Y_n\le x)=\sum_{0\le i\le n-1} (1-e^{-2^i x}) R_{i,n}(0).
$$
For $0\le k<n$, $\sum_{0\le i\le n-1} 2^{ik} R_{i,n}(x)$ is a polynomial of degree at most $n-1$ 
which is $2^{ik}$ at $2^i$, $i=0$, \dots, $n-1$; it must then be $x^k$, so
$$
\sum_{0\le i\le n-1} 2^{ik} R_{i,n}(0)=0, \qquad 1\le k\le n-1.
$$
Therefore, if we write, for any $c\in\omega$,
$$
e_c(x):=e^{-x}-\sum_{0\le k\le c} \frac{(-x)^k}{k!},
$$
then
$$
\PP(Y_n\le x)=-\sum_{0\le i\le n-1} e_c(2^i x) R_{i,n}(0), \qquad 0\le c\le n-1.
$$
Let
\begin{equation}
\label{ex0}
{\cal N}:=\prod_{j\ge 1} \frac{2^j}{2^j-1}, \qquad
F_c(x):={\cal N} \sum_{i\ge 0} \frac{(-1)^{i+1} e_c(2^i x)}{(2^1-1)\cdots(2^i-1)}.
\end{equation}
We have
$$
\lim_{n\to\infty} \PP(Y_n\le x)=\lim_{n\to\infty} \sum_{0\le i\le n-1} V_{c,i,n}(x),
$$
where
\begin{eqnarray*}
V_{c,i,n}(x)&:=&-e_c(2^i x) \prod_{0\le j\le n-1, j\ne i} \frac{2^j}{2^j-2^i}\\
&=& \left(\prod_{1\le j\le n-1-i} \frac{2^j}{2^j-1}\right) (-1)^{i+1} e_c(2^i x)
\prod_{1\le k\le i} \frac{1}{2^k-1}.
\end{eqnarray*}
Then for all $n>i$,
$$
|V_{c,i,n}(x)|\le U_{c,i}(x):=\left(\prod_{j\ge 1} \frac{2^j}{2^j-1}\right)|e_c(2^i x)|\prod_{1\le k\le i} \frac{1}{2^k-1},
$$
and $\sum_{i\ge 0} U_{c,i}(x)$ converges, so by the dominated convergence theorem,
\begin{eqnarray*}
\lim_{n\to\infty} \PP(Y_n\le x)&=&\sum_{i\ge 0} \lim_{n\to\infty} V_{c,i,n}(x)\\
&=& F_c(x).
\end{eqnarray*}
(If $x\in\RR_{<0}$, we define $F_c(x):=0=\lim_{n\to\infty} \PP(Y_n\le x).$)

Since $Y_n\to Y_\infty$ a.s., $Y_n$ also converges to $Y_\infty$ in distribution, so $\PP(Y_\infty\le x)=F_c(x)$
at every continuity point of $\PP(Y_\infty\le x)$.  Since $F_c(x)$ is continuous, 
it must equal $\PP(Y_\infty\le x)$ everywhere.

Let $c$ and $x$ be positive.
If, for complex $z$, we define $S_c(z)$ by
$$
e^{cz}=\sum_{0\le k\le c} \frac{(cz)^k}{k!}+\frac{(cz)^c}{c!} S_c(z),
$$
and $z$ has negative real part, then
\cite[Theorem]{buck1963}
\begin{equation}
\label{ex1}
S_c(z)=\frac{z}{1-z}+O(\frac{1}{c}).
\end{equation}
In this section only, 
by $O(f(\cdots))$ for some function $f$, we mean any quantity for which there is an absolute constant 
$\cal L$ so that its absolute value is no larger than ${\cal L} f(\cdots)$ for all values of the parameters 
of $f$.
Substituting (\ref{ex1}) and
$$
e_c(2^i x)=\frac{(-2^i x)^{c}}{c!} S_{c}(-\frac{2^i x}{c})
$$
into (\ref{ex0}) gives 
\begin{eqnarray*}
F_c(x)&=&{\cal N}^2 \sum_{i\ge 0} (-1)^{i+1} 2^{-i(i+1)/2} e_c(2^i x) \prod_{j>i} (1-2^{-j})\\
&=& {\cal N}^2 \sum_{i\ge 0} (-1)^{i} \frac{(-2^i x)^c}{c!} 2^{-i(i+1)/2} (\frac{2^i x}{c+2^i x}+O(\frac{1}{c}))(1+O(2^{-i})).
\end{eqnarray*}
Substituting $x:=cy/2^c$, $i:=j+c$, we get
\begin{eqnarray*}
F_c(x)&=&{\cal N}^2 \frac{(cy)^c}{c!} 2^{-c(c+1)/2} \sum_{j\ge -c} (-1)^j 2^{-j(j+1)/2} (\frac{2^j y}{2^j y + 1} + O(\frac{1}{c}))(1+O(2^{-(j+c)}))\\
&=&{\cal N}^2 \frac{(cy)^c}{c!} 2^{-c(c+1)/2} (O(\frac{1}{c})+\sum_{j\ge -c} (-1)^j 2^{-j(j+1)/2} \frac{2^j y}{2^j y + 1}),
\end{eqnarray*}
and since removing the lower limit at $-c$ changes the sum by only $O(2^{-c(c+1)/2})$, we have
\begin{eqnarray}
\PP(Y_{\infty}\le x)&=&{\cal N}^2 \frac{(cy)^c}{c!} 2^{-c(c+1)/2} ({\cal P}(y) + O(\frac{1}{c})),
\ \ \hbox{where}\nonumber\\
{\cal P}(z)&:=&\sum_{j\in\ZZ} (-1)^j 2^{-j(j+1)/2} \frac{2^j z}{2^j z+1},\qquad z\in\CC.\label{aseq0}
\end{eqnarray}
For small $\epsilon>0$, the region ${\cal D}(\epsilon):=\{z\in\CC\mid |z+1|\le\epsilon |z|\}$
is a small disk containing $-1$.  On $\CC\setminus\{0\}$ with ${\cal D}(\epsilon)$ and its scalings
by $2^j$ ($j\in\ZZ$) removed, the sum defining ${\cal P}(z)$ converges uniformly, so it is analytic.
Letting $\epsilon\to 0$, it follows that ${\cal P}(z)$ is analytic on $\CC\setminus\{0, -2^j\mid j\in\ZZ\}$;
similarly, it has simple poles at $-2^{j}$ ($j\in\ZZ$.)
Where $\cal P$ is defined, we have
\begin{eqnarray*}
{\cal P}(z)&=& z \sum_{j\in\ZZ} (-1)^j 2^{-j(j+1)/2} \frac{2^j}{2^j z+1}\\
&=& z \sum_{j\in \ZZ} (-1)^j 2^{-(j-1)j/2} \frac{1}{2^j z+1} \\
&=& z \sum_{j\in \ZZ} (-1)^j 2^{-(j-1)j/2} \left( \frac{1}{2^j z+1}-1\right),\\
&\ & \qquad \hbox{since $\sum_{j\in\ZZ} (-1)^j 2^{-(j-1)j/2}=0$}\\
&=& z \sum_{j\in \ZZ} (-1)^{j-1} 2^{-(j-1)j/2} \frac{2^{j-1}\cdot 2 z} {2^{j-1} \cdot 2z + 1}\\
&=& z {\cal P}(2z).
\end{eqnarray*}
Now, set
\begin{equation} \label{ten}
{\cal Q}(z):= 2^{z(z-1)/2} {\cal P}(2^z);
\end{equation}
then ${\cal Q}(z)$ is analytic on $\CC\setminus(\ZZ+(2\pi\iota/\log 2)\ZZ+\pi\iota/\log 2)$,
has simple poles at $\ZZ+(2\pi\iota/\log 2)\ZZ+\pi\iota/\log 2$, and, where ${\cal Q}$ is defined,
\begin{eqnarray*}
{\cal Q}(z+1)&=&2^{(z+1)z/2} {\cal P}(2^{z+1})= 2^{z(z-1)/2} \cdot 2^z {\cal P}(2\cdot 2^z)={\cal Q}(z)
\ \ \hbox{and}\\
{\cal Q}(z+\frac{2\pi\iota}{\log 2})&=&2^{(z+(2\pi\iota/\log 2))(z+(2\pi\iota/\log 2)-1)/2} {\cal P}(2^z)
=-e^{2\pi \iota z} e^{-2\pi^2/\log 2} {\cal Q}(z).
\end{eqnarray*}
However, if $\theta_4(z,q)$ is the theta function \cite[\S 21.1, \S 21.11, \S 21.12]{whit1927}
$$
\theta_4(z,q):=1+2\sum_{i\ge 1} (-1)^i q^{i^2} \cos (2iz), \qquad
q=e^{\pi \iota \tau}, \ |q|<1, \ z, q, \tau\in\CC,
$$
then, for fixed $q$, $\theta_4(z,q)$ is analytic on all of $\CC$ and has zeroes 
at $\pi\ZZ+\pi\tau\ZZ+\frac 12 \pi\tau$, and
$$
\theta_4(z+\pi,q)=\theta_4(z,q), \qquad \theta_4(z+\pi\tau, q)=-\frac{e^{-2\iota z}}{q} \theta_4(z,q).
$$
If we set $\tau:=2\pi\iota/\log 2$, $q:=\exp -2\pi^2/\log 2$, then,
${\cal Q}(z) \theta_4(\pi z, q)$
will be analytic on $\CC$ and doubly periodic, so it is constant and 
$$
{\cal Q}(z)=\frac{\cal K}{\theta_4(\pi z, q)}
$$
for some ${\cal K}\in \CC$, which is real and positive since both ${\cal Q}(0)={\cal P}(1)$ and $\theta_4(0, q)$
are real and positive.  Therefore, ${\cal Q}(r)$ cannot be zero for $r\in\RR$, and since ${\cal Q}(r)$
is then real, it must always be real and positive for $r\in\RR$, so ${\cal P}(r)$ is also always 
real and positive for $r\in\RR_{>0}$.  Also, since $\theta_4(\pi z, q)$ is even, ${\cal Q}(z)$ is even.

Supposing now that $x=c'/2^{c'}$ for some real $c'\ge 1$, we may let $c:=\floor{c'}\ge 1$.  Then
\begin{equation} \label{tenb}
y=\frac{c'/c}{2^{c'-c}}=2^{-\{c'\}}(1+\frac{\{c'\}}{c})
\end{equation}
so $\frac 12<y< 2$, and, from (\ref{aseq0}) and (\ref{ten}),
\begin{eqnarray*}
\PP(Y_{\infty}\le x)&=&{\cal N}^2 \frac{(cy)^c}{c!} 2^{-c(c+1)/2} ({\cal P}(y) + O(\frac{1}{c}))\\
                   &=&{\cal N}^2 \frac{(cy)^c}{c!} 2^{-c(c+1)/2} {\cal P}(y)(1 + O(\frac{1}{c}))\\
                   &=&{\cal N}^2 \frac{(cy)^c}{c!} 2^{-c(c+1)/2} y^{(1-\log_2 y)/2}
                      {\cal Q}(\log_2 y)(1 + O(\frac{1}{c})).
\end{eqnarray*}
After some simplification, using Stirling's approximation, (\ref{tenb}), 
${\cal Q}(\log_2 y)={\cal Q}(-\{c'\})(1+O(1/c))$, 
$O(1/c)=O(1/c')$, and the periodicity and evenness of $\cal Q$, we get
\begin{theorem}
\label{thmasymp}
For real $c'\ge 1$,
$$
\PP(Y_\infty\le \frac{c'}{2^{c'}})=
\frac{{\cal N}^2}{\sqrt{2\pi c'}} e^{c'} 2^{-c'(c'+1)/2} {\cal Q}(c')(1+O(\frac{1}{c'})).
$$
\end{theorem}
It will be convenient later to find $\PP(Y_\infty\le c''y/2^{c''})$, where $c''$ and $y$ are
positive real and $(\log_2 y)^4\le c''$.  We need then to find $c'$ with
\begin{equation}
\label{eqquot}
\left.\frac{c'}{2^{c'}}\right/\frac{c'' y}{2^{c''}}=1,
\end{equation}
and if we set
\begin{equation}
\label{eqcpcpp}
c':= c''-\log_2 y-\frac{\log_2 y}{c''\log 2}+\frac{K}{c''^{3/2}},
\end{equation}
we can verify that, if $c''\ge 6$, the left-hand side of (\ref{eqquot})
is less than 1 if $K=4$ and bigger than 1 if $K=-7$, so
(\ref{eqquot}) must be satisfied with some $-7\le K\le 4$.
Substituting (\ref{eqcpcpp}) into Theorem \ref{thmasymp}, we get
\begin{proposition}
\label{thm12}
For real $c''\ge 6$ and $2^{-c''^{1/4}}\le y\le 2^{c''^{1/4}}$, 
$$
\PP(Y_\infty\le \frac{c'' y}{2^{c''}})=
\frac{{\cal N}^2}{\sqrt{2\pi c''}} (ey)^{c''} 2^{-c''(c''+1)/2} y^{(1-\log_2 y)/2} {\cal Q}(c''-\log_2 y)
(1+O(\frac{1}{\sqrt{c''}})).
$$
\end{proposition}
Finally, we observe that
\begin{eqnarray*}
\PP(Y_\infty>x)&=&1-F_0(x)\\
&=&1-{\cal N} \sum_{i\ge 0} 
\frac{(-1)^{i+1} (e^{-2^i x}-1)}{(2^1-1)\cdots(2^i-1)}\\
&=&1+{\cal N}\sum_{i\ge 0} \frac{(-1)^{i+1}}{(2^1-1)\cdots(2^i-1)}
+{\cal N} \sum_{i\ge 0} \frac{(-1)^i e^{-2^i x}}{(2^1-1)\cdots(2^i-1)}.
\end{eqnarray*}
The last term is an alternating series whose terms decrease in magnitude, so its value 
is between 0 and the first term of the series, which is ${\cal N} e^{-x}$.  
Since $\PP(Y_\infty>x)$ must approach 0 as $x\to\infty$,
the first two terms must cancel, so we have
\begin{proposition}  \label{thm13} For $x\in\RR_{>0}$, 
$$\PP(Y_\infty>x)\le {\cal N} e^{-x}.$$
\end{proposition}

\section{The threshold of the sequence of paths, II}

Returning to the situation of \S 4, we now 
let $n$ be large and
set $L:=\floor{\log_2 n}$, $c'':=\sqrt{\log_2 n}$,
$p:=(c''+\sqrt{c''})/(e 2^{c''})$.  If we use
Proposition \ref{thm12} to estimate $q''$, we will take 
$y:=2(1-(r/2^L))\lambda(1+c''^{-1/2})/(ep)$.
Recalling that $r=(2\log n)/p$ and, since $\lambda=-\log(1-p)$,
$|(\lambda/p)-1|\le p$ for all $0\le p\le \frac 12$,
we find
that $\frac 12\le y\le 1$ for all sufficiently large $n$ and so
$$
q''=\Theta(\frac{1}{\sqrt{c''n}} 2^{c''/2} \Delta^{c''}), \qquad 
\hbox{where\ \ } \Delta:=(1 - \frac{r}{2^L})\frac{\lambda}{p}(1+\frac{1}{\sqrt{c''}}).
$$
Then
$$
pq''^2 \floor{\frac{n}{2L+1}}=\Theta(\frac{\Delta^{2 c''}}{L})
=\Theta(\frac{e^{2\sqrt{c''}}}{L})
$$
will approach infinity as $n\to\infty$, 
so our random distribution
is solvable with a probability that approaches 0 as $n\to\infty$.  This means that, for this choice of $p$,
$$n(\frac{1}{p}-1)=e\frac{2^{\sqrt{\log_2 n}}}{\sqrt{\log_2 n}} n(1 + O(\frac{1}{(\log n)^{1/4}}))$$
is eventually below the geometric threshold of the sequence of families of solvable distributions of the $n$-paths.

We now need to find an upper bound for the geometric threshold.
We first prove a preliminary lemma.
\begin{lemma}  
\label{lemnew}
If $L\in\omega$, $0<p<1$ and $r\in \RR_{\ge 0}$, 
define 
$$
{\cal X}(L,p,r):=\PP(Z_1+\frac{Z_2}{2}+\cdots+\frac{Z_L}{2^{L-1}}<r),
$$
where $Z_1$, \dots, $Z_L$
are i.i.d.~geometric random variables with parameter $p$.
Then
\begin{equation}
\label{xest}
{\cal X}(L,p,r)\le \PP(Y_\infty < (r+3)\lambda) + {\cal N} \exp -2^L \lambda,
\end{equation}
where $\lambda:=-\log(1-p).$
\end{lemma}
\begin{proof}
As in \S 4, let $W_1$, $W_2$, {\dots} be i.i.d.~standard exponential random variables;
then we can realize each $Z_i$ as $\floor{W_i/\lambda}$, so
\begin{eqnarray*}
{\cal X}(L,p,r)&=&\PP(\floor{W_1/\lambda}+\cdots+\frac{\floor{W_L/\lambda}}{2^{L-1}}<r)\\
&\le& \PP((W_1/\lambda)+\cdots+\frac{W_L/\lambda}{2^{L-1}}<r+2)\\
&\le& \PP(\sum_{i\ge 1} \frac{W_i}{\lambda 2^{i-1}}<r+3)+
\PP(\sum_{i\ge L+1} \frac{W_i}{\lambda 2^{i-1}}>1)\\
&=& \PP(Y_\infty<(r+3)\lambda)+\PP(Y_\infty>2^L\lambda)\\
&\le& \PP(Y_\infty<(r+3)\lambda)+{\cal N} \exp -2^L\lambda,
\end{eqnarray*}
by Proposition \ref{thm13}.
\end{proof}
Suppose now as before that we place $Z_i$ pebbles on each vertex $i$, 
where the $Z_i$'s are i.i.d.~geometric random variables with parameter
$p$, and let $M\le L$ be positive integers with $n\ge 2(L+M)+1$.  If
$i\le L+M$, if $i$ is unpebblable, then $\sum_{0\le k\le L-1} Z_{i+k} 2^{-k}$
must be less than 1, and if $i\ge n-(L+M)+1$, if $i$ is unpebblable, then
$\sum_{0\le k\le L-1} Z_{i-k} 2^{-k}$ must be less than 1.  In both 
cases, then, since the $Z_i$'s are i.i.d.,
$$
\PP(\hbox{$i$ unpebblable})\le {\cal X}(L,p,1).
$$
If $L+M+1\le i \le n-(L+M)$, we observe that for $i$ to be unpebblable,
$Z_{i-M}$, \dots, $Z_{i+M}$ must all be less than $2^M$, and
$\sum_{0\le k\le L-1} Z_{i+M+1+k} 2^{-k}$ and 
$\sum_{0\le k\le L-1} Z_{i-M-1-k} 2^{-k}$ must both be less than $2^{M+1}$.
In this case, then,
$$
\PP(\hbox{$i$ unpebblable})\le \PP(Z_1<2^M)^{2M+1} {\cal X}(L,p,2^{M+1})^2.
$$
Summing these probabilities, and using
$$
\PP(Z_1<2^M)= \sum_{0\le j<2^M} \PP(Z_1=j) \le 2^M p,
$$
we find that the probability that our random distribution is 
unsolvable is no more than
\begin{eqnarray}
&\ &2(L+M){\cal X}(L,p,1)+(n-2(L+M))(2^M p)^{2M+1} {\cal X}(L,p,2^{M+1})^2\nonumber\\
&\le& 4L {\cal X}(L,p,1)+n (2^M p)^{2M+1} {\cal X}(L,p,2^{M+1})^2.\label{upper1}
\end{eqnarray}
We now let $n$ be large, set $L:=\floor{\log_2 n}$, $M:=\floor{(\log_2 n)^{1/16}}$, 
$c'':=\sqrt{\log_2 n}$, and $p:=(c''-\sqrt{c''})/(e 2^{c''})$,
and estimate ${\cal X}(L,p,r)$ using (\ref{xest}).
To estimate the first term in (\ref{xest}), we use
Proposition \ref{thm12}, setting $y:=(r+3)\lambda(1-c''^{-1/2})/(ep)$.
In the case $r=1$ we have $1\le y\le 2$ for all sufficiently large $n$
so
\begin{equation}
\label{upper2}
{\cal X}(L,p,1)=\Theta(\frac{1}{\sqrt{c''n}} 2^{3 c''/2} \Delta'^{c''})+O(e^{-\sqrt{n}}),
\ \ \hbox{where} \ \Delta' := \frac{\lambda}{p} (1-\frac{1}{\sqrt{c''}}).
\end{equation}
In the case $r=2^{M+1}$, $\log_2 y = O(M)$ so
\begin{equation}
\label{upper3}
{\cal X}(L,p,2^{M+1})=\Theta(\frac{1}{\sqrt{c''n}} (2^{M+1}+3)^{c''} 2^{-c''/2} 
2^{O(M^2)} \Delta'^{c''})
+O(e^{-\sqrt{n}}).
\end{equation}
Combining (\ref{upper1}), (\ref{upper2}), (\ref{upper3}), and
$\Delta'^{c''}=\Theta(e^{-\sqrt{c''}})$
shows that our random distribution is unsolvable with a probability that
approaches 0 at $n\to\infty$, so,
for this choice of $p$,
$$n(\frac{1}{p}-1)=e\frac{2^{\sqrt{\log_2 n}}}{\sqrt{\log_2 n}} n(1 + O(\frac{1}{(\log n)^{1/4}}))$$
is eventually above the geometric threshold of the sequence of families of solvable 
distributions of the $n$-paths.  Together with our previous lower bound
on the geometric threshold, this proves
\begin{theorem}
\label{penult}  For all positive integers $n$, let the $n$-path 
have $n$ vertices, 1, \dots, $n$, and edges between vertices $i$ and $i+1$
for $i=1$, \dots, $n-1$.  Then the geometric threshold of the
sequence of families of solvable distributions of the $n$-paths is
$$
e\frac{2^{\sqrt{\log_2 n}}}{\sqrt{\log_2 n}} n(1 + O(\frac{1}{(\log n)^{1/4}})).$$
\end{theorem}
\begin{corollary}
The pebbling threshold of the sequence of $n$-paths is
$$
\Theta(\frac{2^{\sqrt{\log_2 n}}}{\sqrt{\log_2 n}} n).
$$
\end{corollary}
\begin{proof} 
Use Theorem \ref{penult}, Theorem \ref{thm10} and 
\cite[Theorem 1.3]{bek2003}.
\end{proof}

\section{One-ended path estimates}

\begin{lemma}  
\label{lembp1}
Let $H$ be a graph which contains $L\ge 1$ vertices, $v_1$, \dots, $v_L$,
such that $d(v_1, v_i)=i-1$ for all $i=2$, \dots, $L$.
There exists some absolute constant $G_{-}\ge 3$ such that, if $g\ge G_{-}$
and an independent, geometrically distributed number of pebbles with parameter
$p:=(1+(2^{\sqrt{2 \log_2 g}}e/(2(1+(\log_2 g)^{-1/4})\sqrt{\log_2 g})))^{-1}$
is placed on each of $v_1$, \dots, $v_L$, then, with probability at least $2/g$, 
$v_1$ is unpebblable, provided that, for the set of vertices $V$ of $H$ apart from
$v_1$, \dots, $v_L$,
\begin{equation}
\label{ehyp}
\sum_{x\in V} {\cal Z}(x) 2^{-d(x,v_1)}\le \frac{2e}{\sqrt{\log_2 g}}.
\end{equation}
\end{lemma}
\begin{proof} This is similar to the first half of the proof of Theorem \ref{penult}.
Let $Z_i$ be the number of pebbles on $v_i$, $i=1$, \dots, $L$. 
The quantity $$Q:=\sum_{\hbox{\small $x$ a vertex of $H$}} {\cal Z}(x) 2^{-d(x,v_1)}$$ is at least 1 if there is a pebble
on $v_1$, and it cannot be increased by pebbling moves.  It follows that $v_1$ will be unpebblable
provided that $Q<1$, which, by (\ref{ehyp}), will certainly be true if
$\sum_{1\le i\le L} Z_i 2^{-(i-1)} < 1-(2e/\sqrt{\log_2 g})$; so, if we let $Z_{L+1}$, $Z_{L+2}$, \dots be 
additional independent geometric random variables with parameter $p$, $v_1$ will be be 
unpebblable if $\sum_{i\ge 1} Z_i 2^{-(i-1)} < 1-(2e/\sqrt{\log_2 g})$.
As in \S 4, we can now set $Z_i:=\floor{W_i/\lambda}$, where $\lambda:=-\log(1-p)$
and $W_1$, $W_2$, {\dots} are i.i.d.~standard exponential random variables,
so it suffices for unpebblability that $$Y_{\infty} < \lambda(1-\frac{2e}{\sqrt{\log_2 g}}).$$
We can compute the probability $q$ of this event using Proposition \ref{thm12}, setting
$$c'':=\sqrt{2 \log_2 g},\ \ 
p':=(p^{-1}-1)^{-1}=\frac{2\sqrt{\log_2 g}}{e 2^{\sqrt{2 \log_2 g}}} (1 + (\log_2 g)^{-1/4}), $$
$$ y:=\frac{\sqrt{2}}{e}\Delta, \ \ \ \Delta:= \frac{\lambda}{p} \frac{1}{p' + 1} (1 + (\log_2 g)^{-1/4}) (1-\frac{2e}{\sqrt{\log_2 g}}).$$
For large $g$, $\Delta$ will be close to 1, so
after choosing $G_{-}$ appropriately, $y$ will be between $\frac 12$ and 1.
According then to the proposition, if we choose $G_{-}$ so as to make $c''$ sufficiently large,
there is some positive constant $C$ such that
$$q\ge C (ey)^{c''} 2^{-c''(c''+1)/2}/\sqrt{c''},$$
or such that $q\ge C \Delta^{c''} /(g\sqrt{c''})$.
Now $\Delta$ is a function only of $g$ and,
for large $g$,
$\log \Delta=2^{1/4} c''^{-1/2} + O(c''^{-1})$,
so choose $G_{-}$ large enough to ensure that 
$\Delta^{c''}/\sqrt{c''} \ge 2/C$.
\end{proof} 

\begin{lemma}  
\label{lembp2}
Let $H$ be a graph which contains a path with $L\ge 1$ vertices, $v_1$, \dots, $v_L$.
Then there exists some absolute constant $G_{+}\ge 3$ such that, if $g\ge G_{+}$
and an independent, geometrically distributed number of pebbles with parameter
$p:=(1+(2^{\sqrt{2 \log_2 g}}e/(2(1-(\log_2 g)^{-1/4})\sqrt{\log_2 g})))^{-1}$
is placed on each of $v_1$, \dots, $v_L$, then 
(A) if $L\ge 1.1\sqrt{2 \log_2 g}$, with probability at least $1-1/(4g)$,
$v_1$ is pebblable, and (B) if
$$
2.2\sqrt{2 \log_2 g}\le L\le \exp (2 \log_2 g)^{1/4},
$$
with probability at least $1-1/(4g)$,
all of $v_1$, \dots, $v_L$ are pebblable.
\end{lemma}
\begin{proof} This is similar to the second half of the proof of Theorem \ref{penult}.
We start with (A).
Let $Z_j$ be the number of pebbles on $v_j$, $j=1$, \dots, $L$.
Set $M:=\floor{(\log_2 g)^{1/16}}$.  Since $L\ge 1.1\sqrt{2 \log_2 g}$, we can choose 
$G_{+}$ large enough to ensure that $L\ge M+1$.  
For $v_1$ to be unpebblable, we must have $Z_{j}<2^M$, $j=1$, \dots, $M+1$, and
$\sum_{0\le j\le L-M-2} Z_{M+2+j} 2^{-j} < 2^{M+1}$;
since the $Z_j$'s are independent and geometrically distributed with 
parameter $p$, 
this will have probability no more than
$(2^Mp)^{M+1} X$, where $X:={\cal X}(L-M-1,p,2^{M+1}),$
and by Lemma \ref{lemnew}, $X\le q+X'$, where
$$
q:= \PP(Y_\infty < (2^{M+1}+3)\lambda),\ \ \ 
X':= {\cal N}\exp - 2^{L-M-1} \lambda,
$$
$$
\lambda:=-\log(1-p).
$$
For an appropriate choice of $G_{+}$, we will have 
$X'\le {\cal N} \exp -2^{0.09\sqrt{2 \log_2 g}},$
so by choosing $G_{+}$ large enough, we can force $X'$
to be less than $\frac 18$ when multiplied by $(2^Mp)^{M+1} g$.
To estimate $q$, use Proposition \ref{thm12}, setting
$$c'':=\sqrt{2 \log_2 g},\ \ 
p':=(p^{-1}-1)^{-1}=\frac{2\sqrt{\log_2 g}}{e 2^{\sqrt{2 \log_2 g}}} (1 - (\log_2 g)^{-1/4}), $$
$$ y:=2^{M+1}\frac{\sqrt{2}}{e}\Delta', \ \ \ \Delta':= \frac{\lambda}{p} \frac{1}{p' + 1} (1 - (\log_2 g)^{-1/4}) (1 + \frac{3}{2^{M+1}}).$$
For large $g$, $\Delta'$ will be close to 1, so $2^M\le y\le 2^{M+1}$, and 
by the proposition, if we choose $G_{+}$ appropriately, we will have, for some constant $C>0$,
$$
q\le \frac{C}{g \sqrt{c''}} 2^{(M+1)c''} 2^{-(M-1)M/2} \Delta'^{c''},
$$
so
\begin{equation}
\label{e002}
(2^Mp)^{M+1}qg\le \frac{C}{\sqrt{c''}} (\frac{\sqrt{2}}{e} c'')^{M+1} 2^{M(M+3)/2} \Delta''^{M+1} \Delta'^{c''},
\end{equation}
where
$$
\Delta'':= \frac{1}{p'+1} (1-(\log_2 g)^{-1/4}).
$$
The logarithm of the right-hand side of (\ref{e002})
is $$-2^{1/4} c''^{1/2} + \frac{M(M+3)\log 2}{2} + O(M \log \log g),$$ 
so we can choose $G_{+}$ so that the right-hand side of (\ref{e002}) is less than $\frac 18$.

For (B), it will suffice to show that for each $i=1$, \dots, $L$, $v_i$
is unpebblable with probability no more than $1/(4gL)$.  We fix some $i$ and let
$\delta:=1$ if $i\le L/2$, $\delta:=-1$ if $i>L/2$; we now try to move pebbles 
onto $v_i$ from $v_{i+\delta}$, $v_{i+2\delta}$, \dots, $v_{i+\floor{L/2}\delta}$.
The proof is then similar to (A), except that $L-M-1$ is replaced by $\floor{L/2}-M$;
also, since we have assumed that $L\le e^{\sqrt{c''}}$, we must bound $e^{\sqrt{c''}} (2^M p)^{M+1} X$ 
instead of $(2^M p)^{M+1} X$.
\end{proof} 

\section{The bouquet of paths}

For positive $n$ and $L$ and nonnegative $g$ such that $g(L-1)+1\le n$, 
let the graph ${\cal B}_{n,g,L}$ be the graph which has $n$ vertices 
and is made by taking $g$ paths with $L$ vertices each and a complete graph, choosing one vertex from
the complete graph and one end-vertex from each of the paths, and identifying these $g+1$ vertices into
a single vertex.  Also,
for a graph $H$ with $n>0$ vertices, we define the {\em geometric pebbling threshold} of $H$ to be the 
unique positive real $x$ for which, if an independent, geometrically distributed number of pebbles with 
parameter $(1+x/n)^{-1}$ is placed on each of the vertices of $H$, the probability of the
distribution being solvable is $\frac 12$.  (See Theorem \ref{thm6} for a proof that 
this probability is strictly increasing with $x$ and that this definition is sensible.)

\begin{lemma}
\label{lemc}
For all $\delta>0$ there is some $m_0=m_0(\delta)$ such that,
if $\alpha\in\RR_{\ge 0}$,
$N$ is a sum of $m$ independent geometric random variables with parameter $(1+\alpha)^{-1}$, 
and both $m$ and $\alpha m$ exceed $m_0$,
then $$\PP(|N-\alpha m|\le \delta \alpha m)\ge \frac 89.$$
\end{lemma}
\begin{proof}  A geometric random variable with parameter $(1+\alpha)^{-1}$ has mean $\alpha$
and variance $\alpha(1+\alpha)$, so $N$ has mean $m\alpha$ and variance $m\alpha(1+\alpha)$.
Then, use Chebyshev's inequality. \end{proof}

\begin{proposition}   
\label{tbouquet1}
There is some integer $G_0\ge 3$ such that if $g\ge G_0$,
$2gL\le n$, and 
$$
\sqrt{2 \log_2 g} \le L - \log_2 n\le \exp (2\log_2 g)^{1/4},
$$
then the geometric pebbling threshold of ${\cal B}_{n,g,L}$ is
$\alpha n$, where
$$
\alpha:=\beta (1+\eta)^{-1}, \ \ 
\beta:=\frac{2^{\sqrt{2\log_2 g}} e}{2 \sqrt{\log_2 g}},\ \ 
\qquad |\eta|\le (\log_2 g)^{-1/4}.
$$
\end{proposition}
\begin{proof}
Let $H:={\cal B}_{n,g,L}$ and let
$\alpha:=\beta (1+\eta)^{-1}$, where $\eta$ is now arbitrary but satisfies
$|\eta|\le (\log_2 g)^{-1/4}$.
We have $L\ge 2$ and
by taking $G_0$ large enough we can ensure that $g\ge G_{-}$, $g\ge G_{+}$,
$\beta\ge 12$, $\alpha\ge 1$, $\log_2 n\ge 1.2 \sqrt{2 \log_2 g}$,
and $\exp (2 \log_2 g)^{1/4}\ge \ceil{2.2 \sqrt{2 \log_2 g}}.$
Suppose that an independent, geometrically distributed number of pebbles with parameter $(1+\alpha)^{-1}$
is placed on each vertex of $H$. 

Set $\eta:=(\log_2 g)^{-1/4}$, consider one of the paths 
$v_1$, \dots, $v_L$ which was identified to make $H$, and let its unidentified end-vertex be $v_1$.  
Let the total number of pebbles on $H$ be $N$, assume that $N\le 2\alpha n$,
and let $V$ be the set of vertices of $H$ apart from $v_1$, \dots, $v_{L-1}$.
Then
$$
\sum_{x\in V} {\cal Z}(x) 2^{-d(x,v_1)}\le 2^{-(L-1)} \sum_{x\in V}{\cal Z}(x)\le N 2^{-(L-1)}
\le 2\alpha n 2^{-(L-1)}
$$
and
$$
2\alpha n 2^{-(L-1)} \le 4 \beta 2^{-\sqrt{2 \log_2 g}} =
\frac{2e}{\sqrt{\log_2 g}},
$$
so we can apply Lemma \ref{lembp1} to this path (with $L$ decreased by 1) to show that 
$v_1$ is unpebblable with probability at least $2/g$.
After doing this 
to each of the paths in $H$ in turn
we can conclude that the probability that the distribution is solvable is
no more than
\begin{equation}
\label{lbd}
\PP(N>2\alpha n) + (1-\frac{2}{g})^g\le \PP(N>2\alpha n)+e^{-2}.
\end{equation}
If we apply Lemma \ref{lemc} (with $m:=n$), then, since $\alpha\ge 1$ and $n\ge 2gL$,
we can choose $G_0$ so that $n$ and $\alpha n$ are forced to be so large that
$\PP(|N-\alpha n|\le\alpha n)\ge \frac 89$.  Since $\frac 19+e^{-2}<\frac 12$,
(\ref{lbd}) then implies that the geometric pebbling threshold of ${\cal B}_{n,g,L}$ is at least 
$\beta n(1+(\log_2 g)^{-1/4})^{-1}$.

Set $\eta:=-(\log_2 g)^{-1/4}$. 
Let $N'$ be the total number of pebbles on the vertices of the complete graph which was identified to make $H$,
and let there be $m$ of these vertices.  If $N'\ge \alpha m/2$, then,
since $m=n-g(L-1)\ge n-gL\ge n/2$,
$$N'\ge \frac{\alpha m}{2}\ge \frac{\alpha n}{4}\ge \frac{\beta n}{4}\ge 3 n.$$
By moving from vertices in the complete graph to any vertex $w$ of the complete graph, we can
then place at least $(3n-m)/2\ge (3n-n)/2=n$ pebbles on $w$.
Now, again consider one of the paths 
$v_1$, \dots, $v_L$ which was identified to make $H$, letting its unidentified end-vertex be $v_1$.  
By moving from vertices in the complete graph to $v_L$, we can place at least $n$ pebbles on $v_L$; 
by moving along the path,
we can then place at least one pebble on $v_{L-j}$, for any $j=1$, \dots, $\ceil{\log_2 n}-1$.
If $L-\ceil{\log_2 n}\ge 2.2 \sqrt{2 \log_2 g}$, we can apply Lemma \ref{lembp2} to the path
with $L$ decreased by $\ceil{\log_2 n}$ and conclude that, with probability at least $1-1/(4g)$,
each of $v_1$, \dots, $v_{L-\ceil{\log_2 n}}$ are pebblable; otherwise, we can apply Lemma \ref{lembp2}
to the path with $L$ replaced by $\ceil{2.2 \sqrt{2 \log_2 g}}$ and conclude that, with probability at least
$1-1/(4g)$, each of $v_1$, \dots, $v_{\ceil{2.2 \sqrt{2 \log_2 g}}}$ are pebblable.
Applying this reasoning to each path, then, the probability that $H$ is solvable is at least
$$
\PP(N'\ge \frac{\alpha m}{2})-g\frac{1}{4g}=\PP(N'\ge \frac{\alpha m}{2})-\frac14.
$$
If we apply Lemma \ref{lemc}, since $m\ge n/2$, we can choose $G_0$ so that
$\PP(|N'-\alpha m|\le\frac 12 \alpha m)\ge \frac 89$.  Since $\frac89 -\frac 14>\frac 12$, this
means that the geometric pebbling threshold of ${\cal B}_{n,g,L}$ is no more than
$\beta n(1-(\log_2 g)^{-1/4})^{-1}$, completing the proof.
\end{proof}

\begin{lemma}
\label{lemfxn}
If we define $\Phi: \RR_{\ge 0}\to \RR_{\ge 0}$ by $$\Phi(\alpha):=\frac{\alpha^2}{2\alpha+1},$$
then $\Phi$ is a strictly increasing bijection, and if $\alpha_2>\alpha_1>0$,
$$
(\frac{\alpha_2}{\alpha_1})^2\ge \frac{\Phi(\alpha_2)}{\Phi(\alpha_1)} \ge \frac{\alpha_2}{\alpha_1}.
$$
\end{lemma}
\begin{proof}  Easy. \end{proof}

The following result is similar to \cite[Theorem 4]{czy2008}.

\begin{proposition}   
\label{tbouquet2}
For any $0<\epsilon<1$, there is some
integer $L_0=L_0(\epsilon)\ge 2$ such that if $g\ge 1$, $2gL\le \epsilon n$,
and $L_0\le L\le (\log_2 n)-L_0$,
then the geometric pebbling threshold of ${\cal B}_{n,g,L}$ is
$\alpha n$, where
$$
\alpha:=\beta (1+\eta), \ \ 
\beta>0\hbox{\ and \ } \Phi(\beta)=\frac{2^{L-1}}{n},\ \ 
\qquad |\eta|\le \epsilon.
$$
\end{proposition}
\begin{proof}
Fix $\epsilon$, let $H:={\cal B}_{n,g,L}$ and 
$\alpha:=\beta (1+\eta)$, where $\eta$ is arbitrary such that $|\eta|\le \epsilon$,
and suppose that an independent, geometrically 
distributed number of pebbles with parameter $(1+\alpha)^{-1}$
is placed on each vertex of $H$.   

Set $\eta:=-\epsilon$ and let $v_1$, \dots, $v_L$ be one of the paths which was identified to make $H$, 
with $v_1$ being the unidentified end vertex.  If $V$ is the set of vertices in $H$ other than
$v_1$, \dots, $v_L$, then 
$$
{\cal Z}(v_1) +\frac{{\cal Z}(v_2)}{2} + 
\cdots+\frac{{\cal Z}(v_L)}{2^{L-1}} + 2^{-(L-1)} \sum_{x\in V} \floor{\frac{{\cal Z}(x)}{2}}
$$
is at least 1 if there is a pebble on $v_1$, and it cannot be increased by pebbling moves.
Arguing as in Lemma \ref{lembp1}, then, the probability that $v_1$ is unpebblable conditioned on the distribution on $V$ is 
at least
$$\PP(Y_\infty<\lambda(1-2^{-(L-1)} N)),$$
where
$$
N:= \sum_{x\in V} \floor{\frac{{\cal Z}(x)}{2}}, \ \ \ \lambda:=\log(1+\alpha^{-1}).
$$
Now, for each $x$, $\floor{{\cal Z}(x)/2}$ is independently 
geometrically distributed
with parameter $1-(1-(1+\alpha)^{-1})^2=(1+\Phi(\alpha))^{-1}$.
Also, $n\ge 2gL\ge 2 L_0$, so if $m$ is the number of vertices in $V$, then
$$m\ge n-gL\ge n(1-\frac{\epsilon}{2})\ge L_0,$$
and, by Lemma \ref{lemfxn},
$$\Phi(\alpha) m
\le \Phi(\beta) (1-\epsilon) m=\frac{2^{L-1}m}{n}(1-\epsilon)\le 2^{L-1}(1-\epsilon)$$
and
$$\Phi(\alpha) m\ge \Phi(\beta)(1-\epsilon)^2 m = \frac{2^{L-1}m}{n}(1-\epsilon)^2\ge 
2^{L_0-1}(1-\epsilon)^2(1-\frac{\epsilon}{2}).$$
Using Lemma \ref{lemc}, then, we can choose $L_0$ large enough so that
$N$ is no more than $2^{L-1} (1-(\epsilon/2))$
with probability at least $\frac 89$, so the probability that $v_1$ is unpebblable is at least
\begin{equation}
\label{e3}
\frac{8}{9} \PP(Y_\infty<\frac{\lambda\epsilon}{2}).
\end{equation}
Since $L\le (\log_2 n)-L_0$, we have
$\Phi(\beta)\le 2^{-L_0-1}$, so we may choose $L_0$ large enough so that $\beta$ is forced
to be small enough, and $\lambda\epsilon/2$ large enough, so that (\ref{e3}) is greater than $\frac 12$.
This proves that, for an appropriate choice of $L_0$, the geometric pebbling threshold of ${\cal B}_{n,g,L}$
is at least $\beta n (1-\epsilon)$.

Now, set $\eta:=\epsilon$, let $V'$ be the set of vertices on the complete graph which was identified to make $H$,
let $V'$ have size $m'$,
and let $N':=\sum_{x\in V'} \floor{{\cal Z}(x)/2}$.  If $N'\ge 2^{L-1}$, we can place $2^{L-1}$ pebbles on the 
identified vertex in $H$, and from there place at least one pebble on any vertex in $H$.  We need then to show
that $\PP(N'\ge 2^{L-1})>\frac 12$.  As before, for each $x$, $\floor{{\cal Z}(x)/2}$ is 
independently geometrically distributed
with parameter $(1+\Phi(\alpha))^{-1}$, so since
$$m'\ge n-gL\ge n(1-\frac{\epsilon}{2})\ge L_0$$
and, by Lemma \ref{lemfxn},
\begin{eqnarray*}
\Phi(\alpha) m'\ge \Phi(\beta)(1+\epsilon)m'\ge 2^{L-1}(1+\epsilon)(1-\frac{\epsilon}{2})
&=& 2^{L-1}(1+ \frac{\epsilon(1-\epsilon)}{2})\\
&\ge& 2^{L_0-1},
\end{eqnarray*}
we can choose $L_0$ large enough so that, by Lemma \ref{lemc}, 
$\PP(N'\ge 2^{L-1})\ge \frac 89$.  This proves that the geometric pebbling threshold 
of ${\cal B}_{n,g,L}$ is no more than $\beta n (1+\epsilon)$, completing the proof.
\end{proof}

\section{The pebbling threshold spectrum}

\begin{lemma}
\label{distlem}
Let $H$ be a connected graph with $n\ge 2$ vertices, let $v$ be a vertex of $H$ such that
all vertices of $H$ are within distance $d\in\ZZ$ of $v$, $d\ge 2$, and let each vertex of $H$ have
a number of pebbles which is independently geometrically distributed with parameter $(1+\alpha)^{-1}$,
$\alpha\in\RR_{>0}$.  Then 
$v$ is unpebblable
with probability at most
$$
\left(
\frac{
e(2^{d-1}+\ceil{n^{1/d}-1}-1)
}{
\ceil{n^{1/d}-1}(1+\Phi(\alpha))
}
\right)^{\ceil{n^{1/d}-1}}.
$$
\end{lemma}
\begin{proof}
For $i=0$, 1, \dots, let $D_i$ be the number of vertices at distance $i$
from $v$ and let $D'_i:=\sum_{0\le j\le i} D_j$.  Since $\log D'_0=0$ and 
$\log D'_d=\log n$, there must be some $0\le i\le d-1$ for which 
$(\log D'_{i+1})-(\log D'_i)\ge (\log n)/d$, 
and then $D_{i+1}/D'_i = (D'_{i+1}/D'_i)-1\ge n^{1/d}-1$.  Since $D'_i\ge D_i$, 
$D_{i+1}/D_i$ is also at least $n^{1/d}-1$.  This means that there must be some vertex $w$ at
distance $i$ from $v$ which has at least $\ceil{n^{1/d}-1}$ neighbors at distance $i+1$.  Letting a set
of $\ceil{n^{1/d}-1}$ of these neighbors be $V$, $v$ will be pebblable if
\begin{equation}
\label{e301}
\sum_{x\in V} \floor{{\cal Z}(x)/2}\ge 2^{d-1},
\end{equation}
since if so we can move $2^{d-1}$ pebbles to
$w$ and then place a pebble on $v$.  Each $\floor{{\cal Z}(x)/2}$ is independently
geometrically distributed with
parameter $p:=(1+\Phi(\alpha))^{-1}$, so the probability of (\ref{e301}) is the probability that,
if we flip a coin with success probability $p$, it takes at least $2^{d-1}+\ceil{n^{1/d}-1}$
flips to get $\ceil{n^{1/d}-1}$ successes.  Another way of saying this is that there are no more than
$\ceil{n^{1/d}-1}-1$ successes in $2^{d-1}+\ceil{n^{1/d}-1}-1$ flips, so if (\ref{e301}) is false,
there must be at least $\ceil{n^{1/d}-1}$ successes in this number of flips.  Set 
$$p':=\ceil{n^{1/d}-1}/(2^{d-1}+\ceil{n^{1/d}-1}-1).$$
If $\ceil{n^{1/d}-1}(1+\Phi(\alpha))\le 2^{d-1}+\ceil{n^{1/d}-1}-1$,
then the claimed bound on the probability of unpebblability is 1 or greater
and there is nothing to prove.  We can assume then that
$\ceil{n^{1/d}-1}(1+\Phi(\alpha))>2^{d-1}+\ceil{n^{1/d}-1}-1$;
now $p<p'<1$, so by \cite[Theorem 1]{arr1989}, the probability that (\ref{e301}) is
false is no more than
$$
\exp -(2^{d-1}+\ceil{n^{1/d}-1}-1)\Omega, \ \ \hbox{where\ } 
$$
\begin{equation}
\label{e302}
\Omega:= p' \log \frac{p'}{p} + (1-p') \log \frac{1-p'}{1-p}.
\end{equation}
Since $y\log y\ge y-1$ for all $0<y<1$, $\Omega\ge p' (-1 + \log (p'/p))$.  Together with (\ref{e302}),
this gives the claimed bound.
\end{proof}

\begin{theorem} \label{geoub} There is some $n_0\ge 3$ such that if $H$ is 
a connected graph with $n\ge n_0$ vertices, then the geometric pebbling
threshold of $H$ is no more than
$$
\frac{2^{\sqrt{2\log_2 n}} e}{2 \sqrt{\log_2 n}} (1 - (\log_2 n)^{-1/4})^{-1} n.
$$
\end{theorem}
\begin{proof}
Choose $n_0$ such that $n_0\ge G_{+}$.
Let $p:=(1+(2^{\sqrt{2 \log_2 n}}e/(2(1-(\log_2 n)^{-1/4})\sqrt{\log_2 n})))^{-1}$,
and suppose that an independent, geometrically distributed number of pebbles with parameter
$p$ is placed on each vertex of $H$.  Pick some vertex $v$ of $H$.  It will do to show
that $v$ is pebblable with probability at least $1-1/(4n)$.  If there is some vertex $x$
of $H$ with $d(v,x)\ge 1.1 \sqrt{2 \log_2 n}$, then this follows immediately from
Lemma \ref{lembp2}.  Otherwise, 
apply Lemma \ref{distlem} with $d:=\ceil{1.1 \sqrt{2 \log_2 n}}$.  
For an appropriate choice of $n_0$, this will always show that $v$
is unpebblable with probability at most $1/(4n)$.
\end{proof}

\begin{theorem} \label{geolb}  
There is some $n_1\ge 1$ such that, if $H$ is a graph with $n\ge n_1$
vertices, then the geometric pebbling threshold of $H$ is at least $\sqrt{n\log 2}$.
\end{theorem}
\begin{proof}
A distribution on any graph 
with $n\ge 1$ vertices will not be solvable if no vertex has two or more pebbles
and some vertex has no pebbles.  If the number of pebbles on each vertex is independently geometrically
distributed with parameter $p$, the probability of this event is $q:=(1-(1-p)^2)^n-(p(1-p))^n$.
For large $n$, if $p:=(1+\sqrt{(\log 2)/n})^{-1}=1-\sqrt{(\log 2)/n}+((\log 2)/n)+O(n^{-3/2})$,
then $q=\frac12 + (\log 2)^{3/2} n^{-1/2} + O(n^{-1})$,
which eventually exceeds $\frac 12$.
\end{proof}

\begin{corollary}  If $(H_i)_{i\in\ZZ_{>0}}$ is any sequence of connected graphs such that the
number of vertices in $H_i$ is strictly increasing with $i$, then the sequence has
some pebbling threshold $t(n)$ which is $\Omega(\sqrt{n})$ and 
$O(2^{\sqrt{2 \log_2 n}} n/\sqrt{\log_2 n})$, where $n$ is the number of vertices in a 
graph in the sequence.
\end{corollary}
\begin{proof}  By Theorems \ref{geoub} and \ref{geolb}, for sufficiently large $i$,
the geometric pebbling threshold $T_i$ of $H_i$ satisfies
$$\sqrt{n_i \log 2}\le T_i \le
\frac{2^{\sqrt{2\log_2 n_i}} e}{2\sqrt{\log_2 n_i}}(1 - (\log_2 n_i)^{-1/4})^{-1}n_i,$$ where
$n_i$ is the number of vertices in $H_i$.  Define the function $t(n)$ by
$t(n_i):=T_i$ for each $i\in\ZZ_{>0}$ and $t(n):=n$ if $n$ is not equal to any $n_i$.
Now apply Theorem \ref{thm10} and
argue as in the proof of \cite[Theorem 1.3]{bek2003} to prove that $t(n)$ is a pebbling threshold
for $(H_i)_{i\in\ZZ_{>0}}$.
\end{proof}

\begin{theorem} \label{tbouquet3}
There is some constant $K>1$ such that,
if $n\ge 2$ is an integer and 
$$\sqrt{n} \le t \le \frac{2^{\sqrt{2\log_2 n}}}{\sqrt{\log_2 n}}n,$$
then there is some connected graph $H$ with $n$ vertices whose geometric pebbling
threshold is between $t/K$ and $K t$.
\end{theorem}
\begin{proof}  
Set $L_0:=L_0(\frac 12)$.
We are free to choose an arbitrary connected graph for $H$
for a finite number of values of $n$, 
at the cost of worsening $K$, so we can assume that $n\ge 2^{2 L_0}$ and $n/(4\log_2 n)\ge G_0$.
Then, we will always choose $H$ to be some ${\cal B}_{n,g,L}$.  Set $\beta:=t/n$
and $\beta_c:=2^{\sqrt{2 \log_2 G_0}}e/(2 \sqrt{\log_2 G_0})$.
\begin{enumerate}
\item If $\beta<\beta_c$, $g$ will always be 1.  Let ${\hat L}:=1+\log_2(\Phi(\beta)n)$.
Then
we let $L$ be $L_0$ if ${\hat L}< L_0$, $\floor{\hat L}$ if $L_0\le {\hat L}\le (\log_2 n)-L_0$,
and $\floor{\log_2 n}-L_0$ if ${\hat L}> (\log_2 n)-L_0$.
\item If $\beta\ge \beta_c$, let $g$ be the maximal integer in $G_0$, $G_0+1$, \dots,
$\floor{n/(4\log_2 n)}$ with $2^{\sqrt{2\log_2 g}}e/(2 \sqrt{\log_2 g})\le \beta$.
Let $L$ be $\ceil{(\log_2 n)+\sqrt{2 \log_2 g}}$.
\end{enumerate}
It is straightforward to verify that, regardless of $t$ or $n$,
the geometric pebbling threshold $t'$ of $H$ can then be computed with Proposition \ref{tbouquet1} or 
Proposition \ref{tbouquet2}, and that there is some absolute constant $K>1$ such that
$t'/t$ is always in $[1/K, K]$.
\end{proof}

\begin{corollary} If $t(n)$ is any positive function of integral $n\ge 1$ which is
$\Omega(\sqrt{n})$ and $O(2^{\sqrt{2\log_2 n}}n/\sqrt{\log_2 n})$, then 
there is some sequence of 
connected graphs $(H_n)_{n\in\ZZ_{>0}}$ with pebbling threshold $t(n)$
such that $H_n$ has $n$ vertices for each $n$.
\end{corollary}
\begin{proof}  Let $t'(1):=1$ and for all integral $n\ge 2$, let
$$t'(n):=\min(\max(t(n),\sqrt{n}), 2^{\sqrt{2\log_2 n}}n/\sqrt{\log_2 n}).$$
Let $H_1$ be the 1-vertex graph, which has geometric pebbling threshold $t''(1):=1$, 
and, for each $n\ge 2$, let $H_n$ be the 
connected graph
given by Theorem \ref{tbouquet3} which has $n$ vertices and geometric pebbling threshold $t''(n)$ between
$t'(n)/K$ and $K t'(n)$.  Then, apply Theorem \ref{thm10} and \cite[Theorem 1.3]{bek2003}
to prove that $t''(n)$ is a pebbling threshold of $(H_n)_{n\in\ZZ_{>0}}$.
It follows that $t'(n)$ and $t(n)$ are also pebbling thresholds for $(H_n)_{n\in\ZZ_{>0}}$.
\end{proof}

\end{document}